\documentclass[12pt]{amsart}
\usepackage{amssymb, enumerate, mdwlist}

\evensidemargin\paperwidth
\advance\evensidemargin-\textwidth
\advance\oddsidemargin-1in
\evensidemargin\oddsidemargin

\topmargin\paperheight
\advance\topmargin-\textheight
\advance\topmargin-1in

\theoremstyle{plain}
\newtheorem{theorem}{Theorem}[section]
\theoremstyle{plain}
\newtheorem{proposition}[theorem]{Proposition}
\theoremstyle{plain}
\newtheorem{lemma}[theorem]{Lemma}
\theoremstyle{plain}
\newtheorem{corollary}[theorem]{Corollary}
\theoremstyle{plain}

\theoremstyle{plain}
\newtheorem{mthm}{Theorem}

\newtheorem{mcor}[mthm]{Corollary}
\theoremstyle{definition}
\newtheorem{definition}[theorem]{Definition}
\theoremstyle{remark}
\newtheorem{remark}[theorem]{Remark}
\theoremstyle{remark}
\newtheorem{example}[theorem]{Example}
\theoremstyle{remark}

\title[Banach Expanders]
{Sphere equivalence, Banach expanders, and extrapolation.}
\author{Masato Mimura} \thanks{Supported in part by the Grant-in-Aid for Young Scientists (B), no.25800033 from the JSPS}
\address{Masato Mimura\\
Mathematical Institute, Tohoku University}
\email{mimura-mas@m.tohoku.ac.jp}
\date{\today}

\begin{document}

\begin{abstract}
We study the Banach spectral gap $\lambda_1(G;X,p)$ of finite graphs $G$ for pairs $(X,p)$ of Banach spaces and exponents. We define the notion of sphere equivalence between Banach spaces and show a generalization of Matou\v{s}ek's extrapolation for Banach spaces sphere equivalent to uniformly convex ones. As a byproduct, we prove that expanders are automatically expanders with respects to $(X,p)$ for any $X$ sphere equivalent to a uniformly curved Banach space and for any $p\in (1,\infty)$.
\end{abstract}


\maketitle

\section{Introduction}
In this paper, let $G=(V,E)$ be a finite connected undirected graph, possibly with multiple edges and self-loops (here $E$ is the set of oriented edges). We equip $G$ with the path metric and regard as a metric space. Let $\Delta (G)$ denote the maximal degree of $G$. Let $(X,p)$ be a pair of a Banach space $X$ and an exponent $p$. Our $p$ is always assumed to lie in $[1,\infty)$. Let $\ell_p$ denote $\ell_p(\mathbb{N}, \mathbb{C})$, $L_p$ denote $L_p([0,1],\mathbb{C})$, and $\ell_p^m$, and $\ell_{\infty}^m$ ($m\in \mathbb{N}$) respectively denote the $m$-dimensional real $\ell_p$, and $\ell_{\infty}$ spaces. For $(X,p)$, define $\tilde{X}_{(p)}$ to be the \textit{$p$-stabilization} of $X$, namely, $\tilde{X}_{(p)}:=\ell_p(\mathbb{N},X)$. For a Banach space $X$, $S(X)$ denotes the unit sphere of $X$. We use freely the symbol $a \precsim b$ for two nonnegative functions from the same parameter set $\mathcal{T}$ if there exists a positive multiplicative constant $C > 0$ \textit{independent of $t\in \mathcal{T}$} such that for any $t \in \mathcal{T}$, $a(t) \leq C b(t)$ holds true. We use the symbol $a \asymp b$ if $a \precsim b$ and $a \succsim b$ hold. We use the symbol $a \precsim_q b$ if parameter set $\mathcal{T}$ has  variable $q$ and the positive multiplicative constant $C=C_q$ may depend on the choice of $q$. We write $a \precnsim b$ if $a\precsim b$ holds but $a \succsim b$ fails to be true.

The main topic of this paper is the notion of \textit{$(X,p)$-anders} for a pair $(X,p)$, which is defined in terms of the \textit{Banach spectral gap}. Note that Mendel and Naor \cite{MN} have explicitly introduced the notion of \textit{nonlinear spectral gaps} (for the more general case where $X$ is a metric space) and studied that in detail. 

\begin{definition}\label{DefinitionSpgap}
\begin{enumerate}[$(1)$]
  \item The \textit{Banach spectral gap} is defined as follows: the $(X,p)$\textit{-spectral gap} of $G$, written as $\lambda_1 (G;X,p)$, is 
\[
\lambda_1 (G;X,p) := \frac{1}{2}\inf_{f \colon V\to X} \frac{\sum_{v\in V}\sum_{e=(v,w)\in E}\|f(w)-f(v)\|^p }{\sum_{v\in V}\| f(v)- m(f)\|^p}.
\]
Here $m(f):=\sum_{v\in V}f(v)/|V|$ and $f$ runs over all nonconstant maps. 
 \item A sequence  $\{G_n\}_{n\in \mathbb{N}}$ of finite connected graphs is called a \textit{family of $(X,p)$-anders} $($or simply, \textit{``$(X,p)$-anders"}) if the following three conditions are satisfied: $\sup_{n}\Delta (G_n) < \infty$; $\lim_{n\to \infty}\mathrm{diam}(G_n)=\infty$; and $\inf_{n} \lambda_1(G_n;X,p) >0$.
\end{enumerate}
\end{definition}
If $G$ is regular, then $\lambda_1(G;\mathbb{R},2)$ is the first positive eigenvalue of the nonnormalized combinatorial Laplacian of $G$. Therefore $(\mathbb{R},2)$-anders are \textit{expanders} in the classical sense, see  a survey \cite{HLW}. Being $(X,p)$-anders for some fixed $p$ implies poor embeddability into $X$. See Section~\ref{SectionPre} on \textit{coarse embeddings} and \textit{distortions}.

We recall the definitions of \textit{uniformly curved} Banach spaces, and of \textit{uniformly convex} Banach spaces; and give the definition of the \textit{sphere equivalence} among Banach spaces.

\begin{definition}\label{DefintionUC}
Let $X$ be a Banach space.
\begin{enumerate}[$(1)$]
  \item (Pisier \cite{Pisier}, using some idea of V. Lafforgue.) The $X$ is said to be \textit{uniformly curved} if $\lim_{\epsilon \to +0}\Lambda_X(\epsilon)$$=0$ holds. Here $\Lambda_X(\epsilon)$ denote the infimum over those $\Lambda\in (0,\infty)$ such that for every $n\in \mathbb{N}$, every matrix $T=(t_{ij})_{i,j}\in \mathrm{M}_n(\mathbb{R})$ with 
\[
\|T\|_{L_2^n(\mathbb{R})\to L_2^n(\mathbb{R})}\leq \epsilon \quad \textrm{and}\quad \|\mathrm{abs}(T)\|_{L_2^n(\mathbb{R})\to L_2^n(\mathbb{R})}\leq 1,
\]
where $\mathrm{abs}(T)=(|t_{ij}|)_{i,j}$ is the entrywise absolute value of $T$, satisfies that 
\[
\|T\otimes I_X\|_{L_2^n(X)\to L_2^n(X)}\leq \Lambda.
\]
  \item The $X$ is said to be \textit{uniformly convex} if for any $\epsilon \in (0,2]$,
\[
\sup\{\|x+y\|/2: x,y \in S(X),\ \|x-y\|\geq \epsilon  \}<1.
\]
\end{enumerate}
\end{definition}
Note that Pisier \cite[Section~2]{Pisier} has showed that if $X$ is uniformly curved, then it is \textit{superreflexive}, which is equivalent to saying that $X$ admits an equivalent uniformly convex norm.

\begin{definition}\label{DefinitionSE}
Two Banach spaces $X$ and $Y$ are said to be \textit{sphere equivalent}, written as $X \sim_S Y$, if there exists a uniform homeomorphism (, namely, a biuniformly continuous map) between the two unit spheres $S(X)$ and $S(Y)$. We write $[Y]_S$ for the sphere equivalence class of $Y$.
\end{definition}

One of the motivating big open problems in this field is: whether any (classical) expander family is an $(X,p)$-ander family for each $p\in [1,\infty)$ and every Banach space $X$ which does not admit embeddings of the spaces $\{\ell_{\infty}^m\}_{m\geq 1}$ with uniformly bounded distortions (note that this condition on $X$ is equivalent to saying that $X$ is of \textit{finite} (or equivalently, \textit{nontrivial}) \textit{cotype}). Note that there are \textit{no} $(X,p)$-anders for any $p\in [1,\infty)$ if $X$ has trivial cotype by the Frech\'{e}t classical embeddings. In this paper, we study dependencies of $\lambda_1 (G;X,p)$ on $X$; and on $p$, and prove results on the problem above. The main theorems of this paper roughly state the stabilities of $(X,p)$-ander properties in terms of the sphere equivalence (see Definition~\ref{DefinitionSE}) between Banach spaces.

\begin{mthm}[For more precise statement, see Theorem~$\ref{TheoremSEgap}$]
\label{mthma}

Let $X,Y$ be Banach spaces. If $X \sim_S Y$, then for any $p\in [1,\infty )$, and a sequence  $\{G_n\}_n$, $\{G_n\}_n$ is a family of $(X,p)$-anders if and only if it is a family of $(Y,p)$-anders.
\end{mthm}

\begin{mthm}[Generalization of Matou\v{s}ek's extrapolation]
\label{mthmb}

Let $p,q \in (1,\infty)$. Then for any Banach space $X$ that is sphere equivalent to a uniformly convex Banach space, and a sequence  $\{G_n\}_n$, $\{G_n\}_n$ is a family of $(X,p)$-anders if and only if it is a family of $(X,q)$-anders.
\end{mthm}

We note that recently A. Naor, in Theorem~1.10 and Theorem~4.15 in \cite{Naor}, has independently established similar results. Our approach is \textit{group theoretic}, and different from his. In our proof, we introduce the \textit{``Gross trick"}, see Subsection~\ref{subsection:Gross}.

We review previous works in these directions. On the dependence on $X$,  N. Ozawa \cite{Ozawa} showed that expanders satisfy a weak form of $(X,1)$-ander condition, if $X$ is sphere equivalent to $\ell_2$. G. Pisier \cite[Section~3]{Pisier} showed that expanders are $(X,2)$-anders for uniformly curved Banach spaces $X$. On the dependence on $p$, J. Matou\v{s}ek \cite{Matousek} made the following estimation of $\lambda_1(G;\mathbb{R},p)$ in terms of $\lambda_1(G;\mathbb{R},2)$, which is now called the \textit{Matou\v{s}ek extrapolation} (see also  in \cite[Lemma~5.5]{BLMN}). 

\begin{theorem}[Matou\v{s}ek's extrapolation. \cite{Matousek}]
\begin{enumerate}[$(1)$]
  \item For $p\in [1,2)$, $\lambda_1(G;\mathbb{R},p)\succsim_{\Delta(G),p} \lambda_1(G;\mathbb{R},2)$.
  \item For $p\in [2,\infty)$, $\lambda_1(G;\mathbb{R},p)\asymp_{\Delta(G),p} \lambda_1(G;\mathbb{R},2)^{p/2}$.
\end{enumerate}
In particular, the properties of being $(\mathbb{R},p)$-anders are all equivalent for each $p\in [1,\infty)$.
\end{theorem}

As a corollary to above Theorems A and B, and aforementioned works of Ozawa and Pisier, we have the following corollary. 

\begin{mcor}\label{mcorc}
Any expanders are automatically $(X,p)$-anders for a Banach space $X$ sphere equivalent to uniformly curved Banach space and for $p\in (1,\infty)$. If, moreover, $X \in [\ell_2]_S$, then the assertion above holds even for $p=1$.
\end{mcor}
With the aid of a generalized Grigorchuk--Nowak inequality (see Theorem~\ref{TheoremGriNo} and \cite{GN}), this implies that expanders achieve the \textit{worst order of distortions} into any infinite-dimensional Banach space $X$ that is sphere equivalent to a uniformly curved Banach space. See Corollary~\ref{CorollaryExpDist} for details.

\section{Preliminaries}\label{SectionPre}
We use the notation of the Introduction.
\subsection{Basic properties of Banach spectral gaps}
\begin{lemma}\label{LemmaSpGap}
\begin{enumerate}[$(1)$]
  \item If $Y$ is a  subspace of $X$, then $\lambda_1(G;Y,p)\geq \lambda_1(G;X,p)$.
  \item For a measure space $(\Omega ,\mu)$, we have that $\lambda_1(G;X,p)= \lambda_1(G;L_p(\Omega ,X),p)$. Here $L_p(\Omega ,X)$ denotes the space of all $p$-Bochner integrable maps.

In particular, $\lambda_1(G;\mathbb{C},p)= \lambda_1(G;\ell_p,p) = \lambda_1(G;L_p,p)$, $\lambda_1(G;\mathbb{R},2)=\lambda_1(G;\mathbb{C},2)=\lambda_1(G;\ell_2,2)$, and $\lambda_1(G;\tilde{X}_{(p)},p)=\lambda_1(G;X,p)$ for any $(X,p)$.
  \item We have that $\lambda_1(G;\ell_2,p)= \lambda_1(G;L_p,p) = \lambda_1(G;\ell_p,p)= \lambda_1(G;\mathbb{C},p)=  \lambda_1(G;\mathbb{R},p)$.
\end{enumerate}
\end{lemma}

\begin{proof}
Item $(1)$ is obvious. On item $(2)$, ``$\geq$" is from $(1)$. On the opposite (``$\leq $") direction, observe that for any $0\ne f=(f_\omega)_{\omega\in \Omega}\colon V\to L_p(\Omega,X)$ with $m(f)=0$, 
\[
2\lambda_1(G;X,p)\sum_{v\in V} \|f_{\omega}(v)\|^p \leq \sum_{v\in V}\sum_{e=(v,w)\in E}\|f_{\omega}(v)-f_{\omega}(w)\|^p,
\]
and integrate these inequalities over $\omega\in \Omega$. On item $(3)$, $\lambda_1(G;\ell_2,p)\geq \lambda_1(G;L_p(\mathbb{R}),p)$ holds because $\ell_2$ is isometrically embeddable into $L_p(\mathbb{R})$ (see a paper of L. E. Dor \cite{Dor}) and $\lambda_1(G;L_p(\mathbb{R}),p)= \lambda_1(G;\mathbb{R},p)\geq \lambda_1 (G;\ell_2,p)$ follows from $(1)$ and $(2)$.
\end{proof}

We note that by this lemma together with the Matou\v{s}ek extrapolation, any $(X,p)$-anders are expanders.

\subsection{Coarse embedding of coarse disjoint unions}
For a sequence of finite connected graphs $\{G_n\}_n$, define a \textit{coarse disjoint union} $\coprod_n G_n$ to be an (infinite) metric space $(\coprod_n G_n ,d)$ whose point set is $\bigsqcup_n V_n$ and whose metric satisfies the following conditions.
\begin{itemize}
\item For every $n$, $d\mid_{V_n\times V_n}= d_n$, where $d_n$ denotes the (path) metric on $G_n$.
\item For $n\ne m$, $\mathrm{dist}(V_n, V_m)\geq \mathrm{diam}(G_n)+\mathrm{diam}(G_m)+n+m$.
\end{itemize}

We say $f\colon \coprod_n G_n \to X$ is a \textit{coarse embedding} if $f$ is a Lipschitz map and if moreover there exists a nondecreasing $\rho\colon \mathbb{R}_{\geq 0} \to \mathbb{R}_{\geq 0}$ with $\lim_{t\to +\infty}\rho(t) = +\infty$ such that for any $v,w \in \coprod_n G_n$, $\rho (d(v,w)) \leq \| f(v) - f(w)\|$ holds true. We define $\mathcal{R}_X(\coprod_n G_n)$ to be the class of $\rho$ with respect to which such $f$ exists, and call a member in $\mathcal{R}_X(\coprod_n G_n)$ a \textit{compression function}.

M. Gromov \cite{Gromov} observed that a coarse disjoint union of an expander family does not admit coarse embeddings into $\ell_2$. See a book \cite{BookRoe} for the argument, the main parts of where were known even before Gromov made this observation, see \cite{Matousek}. The argument can be easily adjusted to show the following fact: `` a coarse disjoint union of an $(X,p)$-ander family for some $p\in [1,\infty)$ does not admit coarse embeddings into $X$." A different proof of this fact is given in Corollary~\ref{Corollary:noCE}.

\subsection{Distortions and inequalities}
\begin{definition}
 The \textit{distortion} of $G$ into $X$, denoted by $c_X(G)$ is defined by
\[
(1\leq )c_X(G):=\inf_{f\colon V\to X,\ \mathrm{bi}-\mathrm{Lipschitz}}\|f\|_{\mathrm{Lip}}\|f^{-1}\|_{\mathrm{Lip}}.
\]
\end{definition}

A result of J. Bourgain \cite{Bourgain} states that $c_{\ell_2}(G) \precsim \log |G|$ (in fact, $\ell_2$ may be replaced with $\ell_2^{O((\log |G|)^2)}$). Note that $c_{\ell_p}(G)\leq c_{\ell_2}(G)$ because $\ell_2$ isometrically embeds $L_p$ (see also the proof of Lemma~\ref{LemmaTauConst}). Moreover, Linial--London--Rabinovich \cite{LLR} have generalized the result of Bourgain above and showed that $c_{\ell_p^{O((\log |G|)^2)}}(G) \precsim \log |G|$ for all $p\in [1,\infty)$. Grigorchuk--Nowak \cite{GN} and Jolissaint--Valette \cite{JV} have, respectively, showed inequalities that relate $(\ell_p,p)$-spectral gaps to $\ell_p$ (or $L_p$)-distortions. Here we have straightforward generalizations of these.

\begin{theorem}
[Generalized Grigorchuk--Nowak inequality]
\label{TheoremGriNo}

For any $G,X$, $p\in [1,\infty)$ and  $\epsilon \in (0,1)$, we have that
\[
c_X(G)\geq \frac{(1-\epsilon)^{1/p}r_{\epsilon}(G)}{2}\mathrm{diam}(G) \left(\frac{\lambda_1(G;X,p)}{\Delta(G)}\right)^{1/p}.
\]
Here $r_{\epsilon}(G)$ is defined as $\inf \{\mathrm{diam}(A)/\mathrm{diam}(G):|A|\geq \epsilon |V|\}$. 
\end{theorem}

\begin{theorem}
[Generalized Jolissaint--Valette inequality]
\label{TheoremJoVa}

For any $G,X$, and $p\in [1,\infty)$, we have that 
\[
c_X(G)\geq 2^{-(p-1)/p}D(G) \left(\frac{\lambda_1(G;X,p)}{k(G)}\right)^{1/p}.
\]
Here $k(G)$ is the average degree of $G$ and $D(G)$ is the maximal displacement, namely, $\max_{\alpha \in \mathrm{Perm}(V)}\min_{v\in V}d(v,\alpha\cdot v)$, where $\mathrm{Perm}(V)$ denotes the group of all permutations on $V$. In particular, if $G$ is a vertex-transitive graph, then
\[
c_X(G)\succsim \mathrm{diam}(G) \left(\frac{\lambda_1(G;X,p)}{\Delta(G)}\right)^{1/p}.
\]
\end{theorem}

\begin{proof}
For the proofs of the above two theorems, we need to replace $f$ with $f-m(f)$ for a biLipschitz map $f\colon V\to X$. Then the original proofs in \cite{GN} and \cite{JV} work without any modification. Note that, in their papers, they use a different definition of $\ell_p$-spectral gap. This one is at least $\lambda_1(G;\mathbb{R},p)$ and at most $2^p\lambda_1(G;\mathbb{R},p)$.
\end{proof}

\begin{corollary}\label{CorollaryDist}
Let $\{G_n\}_n$ be $(X,p)$-anders. Then we have $c_X(G_n) \succsim_X \mathrm{diam}(G_n)$. If moreover $X$ is infinite-dimensional, then  $c_{X}(G_n) \asymp_X \mathrm{diam}(G_n)$.
\end{corollary}

\begin{remark}\label{remark:referee}
The referee has pointed out that if $X$ is infinite-dimensional, then by Dvoretzky's theorem, the spaces $\{\ell_2^m\}_{m\geq 1}$ embed into $X$ with uniformly bounded distortions. This comment has improved the statement of Corollary~\ref{CorollaryDist}.
\end{remark}

\begin{proof}
Note that $\{G_n\}_n$ is in particular a family of expanders and is of exponential growth. Hence the first assertion follows from Theorem~\ref{TheoremGriNo}. The second one is from the first one together with the aforementioned result of Bourgain (see also Remark~\ref{remark:referee}). 
\end{proof}

\subsection{Distortions and compression function}
The following lemma is a generalization of a special case of Austin's lemma \cite[Lemma~3.1]{Austin}, which is used in \cite{MS} to obtain some restriction to $\mathcal{R}_X( \coprod_n G_n )$ for certain explicit $\{G_n\}_n$'s. 
\begin{lemma}
\label{LemmaAustin}
Let $\{G_n\}_n$ be a sequence of finite connected graphs with $\mathrm{diam}(G_n)\nearrow \infty$ $($possibly with $\sup_n \Delta(G_n)= \infty)$. Let $\rho \colon \mathbb{R}_+ \nearrow \mathbb{R}_+$ be a map with $\lim_{t\to +\infty }\rho(t) =+\infty$ which satisfies that $\rho(t)/t$ is nonincreasing for $t$ large enough. For any $X$, if for $n$ large enough $\frac{\mathrm{diam}(G_n)}{\rho (\mathrm{diam}(G_n))} \precnsim c_X(G_n)$ 
holds, then we have that $\rho \not \in \mathcal{R}_X( \coprod_n G_n )$.
\end{lemma}

\begin{proof}
Suppose to the contrary that there exists a coarse embedding $f\colon \coprod_n G_n \to X$ such that 
$$
\rho (d(v,w))\precsim \| f(v)-f(w)\|, \quad v,w \in \coprod_n G_n 
$$
holds. Set $f_n:=f\mid_{G_n} \colon V_n \to X$. We may assume that $f$ is a $1$-Lipschitz map and that each $f_n$ is biLipschitz. Then  we have the following order inequalities.
\begin{align*} 
\frac{\mathrm{diam}(G_n)}{\rho (\mathrm{diam}(G_n))} 
&\precnsim c_X(G_n) \leq \|f_n^{-1}\|_{\mathrm{Lip}} 
\leq \max_{v\ne w\in V_n}\frac{d(v,w)}{\|f_n(v)-f_n(w)\|} \\
&\precsim \max_{v\ne w\in V_n}\frac{d(v,w)}{\rho(d(v,w))} 
 \precsim \frac{\mathrm{diam}(G_n)}{\rho(\mathrm{diam}(G_n))}.
\end{align*}
This is a contradiction.
\end{proof}

\begin{corollary}\label{Corollary:noCE}
Let $\{G_n\}_n$ be a sequence of finite connected vertex-transitive graphs $($possibly with $\sup_n \Delta(G_n) =\infty)$. Let $X$ be a Banach space. Assume that $\mathrm{diam}(G_n)\nearrow \infty$ as $n\to \infty$ and that $\inf_n \lambda_1(G_n;X,p)/\Delta(G_n) >0$ for some $p\in [1,\infty)$. Then $\coprod_n G_n$ does not coarsely embed into $X$.
\end{corollary}

\begin{proof}
Suppose to the contrary that $\mathcal{R}_X(\coprod_n G_n) \ne \emptyset$. Then by replacing $\rho$ with smaller function if necessary, we have $\rho \in \mathcal{R}_X(\coprod_n G_n)$ with $\rho(t)/t$ nonincreasing for large $t$. Then Lemma~\ref{LemmaAustin} applies and we would have that $\rho \not \in \mathcal{R}_X(\coprod_n G_n)$. 
\end{proof}

\section{$(\tau)$-type constants and sphere equivalence}\label{SectionTau}
A key to the proof of Theorem~\ref{mthma} for Schreier coset graphs is a certain representation-theoretic constant, as we shall define below. In this section, let $\Gamma$ be a finitely generated group, $S\not \ni e$ be a symmetric finite generating set of $\Gamma$, and $H$ be a subgroup of $\Gamma$ of finite index. By $\mathrm{Sch}(\Gamma, H,S)$ we mean the \textit{Schreier coset graph}, where we take left cosets as vertices.

\subsection{$(\tau)$-type constants}
\begin{definition}
\begin{enumerate}[$(i)$]
  \item Let $(\Gamma, H,S)$ be as in this section.
\begin{enumerate}[$(1)$]
  \item Define $\pi_{X,p}$ as the quasi-regular representation of $\Gamma$ on $\ell_p(\Gamma/H ,\tilde{X}_{(p)})$, namely, for $\gamma \in \Gamma$ and $\xi \in \ell_p(\Gamma/H ,\tilde{X}_{(p)})$, $ \pi_{X,p}(\gamma) \xi (xH):= \xi (\gamma^{-1}xH)$. 
Then $\ell_p(\Gamma/H ,\tilde{X}_{(p)})$ decomposes as $\Gamma$-representation spaces: $\ell_p(\Gamma/H ,\tilde{X}_{(p)})= \ell_p(\Gamma/H ,\tilde{X}_{(p)})^{\pi_{X,p}(\Gamma)}\oplus \ell_{p,0}(\Gamma/H ,\tilde{X}_{(p)})$. 
Here the first space is the space of $\pi_{X,p}(\Gamma)$-invariant vectors, and the second space is the space of ``zero-sum" functions, namely, $\ell_{p,0}(\Gamma/H ,\tilde{X}_{(p)}):=\{\xi \in \ell_{p}(\Gamma/H ,\tilde{X}_{(p)}): \sum_{v \in \Gamma/H} \xi(v) =0\}$.

We use the same  symbol $\pi_{X,p}$ for the restricted representation on $\ell_{p,0}(\Gamma/H ,\tilde{X}_{(p)})$.
  \item ($p$-displacement constant) The $p$\textit{-displacement constant} of $(\Gamma,H,S)$ on $X$, written as $\kappa_{X,p} (\Gamma, H, S)$, is defined as 
\[
\kappa_{X,p}(\Gamma,H,S):= \inf_{0\ne \xi \in \ell_{p,0}(\Gamma /H ,\tilde{X}_{(p)})}\sup_{s\in S}\frac{\| \pi_{X,p}(s)\xi -\xi\|}{\|\xi \|}.
\]
 \end{enumerate}
\item ($(\tau)$-type constant) Let $\Gamma$ be a finitely generated group and $S\not \ni e$ be a symmetric generating set. Then the \textit{$p$-$(\tau)$-type constant} of $(\Gamma, S)$ on $X$, written as $\kappa^{(\tau)}_{X,p}(\Gamma, S)$, is defined by $\kappa^{(\tau)}_{X,p}(\Gamma, S):=\inf_{N}\kappa_{X,p}(\Gamma, N,S)$. Here $N$ runs over all finite index subgroups of $\Gamma$.

\end{enumerate}
We omit writing $p$ if $p$ is fixed.
\end{definition}

The terminology ``$(\tau)$-type constant" is inspired by the relation with \textit{property} $(\tau)$, see \cite{BookLubotzkyZuk}.

\begin{lemma}\label{LemmaTauConst}
Let $(\Gamma,H,S)$ be as in this subsection.
\begin{enumerate}[$(1)$]
  \item Let $X$ be a Banach space and $Y$ be a $($Banach$)$ subspace of $X$. Then for any $p\in [1,\infty)$, $\kappa_{X,p}(\Gamma,H,S) \leq \kappa_{Y,p}(\Gamma,H,S)$ holds.
  \item For any $p\in [1,\infty)$ and any $q\in [1,\infty)$, $\kappa_{\ell_p,q}(\Gamma,H,S) =\kappa_{L_p,q}(\Gamma,H,S)$.
\end{enumerate}
\end{lemma}

\begin{proof}
Item $(1)$ is straightforward. Item $(2)$ follows from an argument (see \cite{JV}) of approximating $L_p$-functions by step functions. 
\end{proof}

The following lemma is well known for the case where $(X,p)=(\ell_2,2)$. We however include the proof for the reader's convenience.

\begin{lemma}\label{LemmaDisp}
For a Schreier coset graph $G=\mathrm{Sch}(\Gamma,H,S)$ and a pair $(X,p)$, we have that 
\[
\kappa_{X,p}(\Gamma ,H,S)^p \leq \lambda_1(G;X,p) \leq \frac{|S|}{2} \kappa_{X,p}(\Gamma ,H,S)^p.
\]
\end{lemma}

\begin{proof}
First note that by Lemma~\ref{LemmaSpGap} $\lambda_1(G;X,p)= \lambda_1(G;\tilde{X}_{(p)},p)$. Take a nonconstant map $f\colon V\to \tilde{X}_{(p)}$ and by replacing $f$ with $f-m(f)$ we may assume $m(f)=0$. Then we may regard $f$ as a nonzero vector $\xi \in \ell_{p,0}(\Gamma/H ,\tilde{X}_{(p)})$. Therefore
\begin{align*}
\lambda_1(G;X,p)&= \frac{1}{2}\inf_{0\ne \xi \in \ell_{p,0}(\Gamma/H ,\tilde{X}_{(p)})} \frac{\sum_{v\in \Gamma/H} \sum_{s^{-1}\in S}\|\pi_{X,p}(s)\xi(v)-\xi(v) \|_{\tilde{X}_{(p)}}^p}{\|\xi\|^p}\\
&= \frac{1}{2}\inf_{0\ne \xi \in \ell_{p,0}(\Gamma/H ,\tilde{X}_{(p)})} \sum_{s\in S}\left(\frac{\|\pi_{X,p}(s)\xi-\xi \|}{\|\xi\|}\right)^p.
\end{align*}
This ends our proof (note that $\|\pi_{X,p}(s)\xi-\xi \|=\|\pi_{X,p}(s^{-1})\xi-\xi \|$).
\end{proof}

If we consider a sequence of Schreier coset graphs with nonuniformly bounded degree, then the estimations in Lemma~\ref{LemmaDisp} do not give the optimal-order inequalities. Pak--\.{Z}uk \cite[Proposition~2]{PZ}, however, have sharpened the left-hand side of the inequalities for $(X,p)=(\ell_2,2)$ if the finite generating set $S$ has an enormous symmetry. Here we present a generalization of their result in general setting.

\begin{theorem}[Generalized Pak--\.{Z}uk theorem]
\label{TheoremPZ}

Let $\Gamma$ be a finitely generated group, $S\not\ni e$ be a symmetric finite generating subset, and $(X,p)$ be a pair. Assume that there exists a finite $S$-preserving subgroup $Q\leqslant \mathrm{Aut}(\Gamma) $. Let us denote by $S_1,\dots ,S_m$ the partition of $S$ into orbits of $Q$. Define $\nu:=\max_{1\leq i \leq m} \frac{|S|}{|S_i|}$. 

Then we have for any subgroup $H$ in $\Gamma$ of finite index that 
\[
\frac{|S|}{2\nu}\kappa_{X,p}(\Gamma ,H,S)^p \leq \lambda_1(\mathrm{Sch}(\Gamma,H, S);X,p) \leq \frac{|S|}{2} \kappa_{X,p}(\Gamma ,H,S)^p.
\]
\end{theorem}

\begin{proof}
Let $N:=|Q|$ and suppose $Q:=\{\sigma_1,\ldots ,\sigma_N \}$. Consider a new isometric $\Gamma$-representation $\pi'$ on $\ell_{p,0}(\Gamma/H, p\textrm{-}\bigoplus_{i=1}^{N}\tilde{X}_{(p)})$ in the following manner ($p\textrm{-}\bigoplus$ means the $\ell_p$-direct sum): for $\gamma \in \Gamma$ and $\eta \in \ell_{p,0}(\Gamma/H, p\textrm{-}\bigoplus_{i=1}^{N}\tilde{X}_{(p)})$, 
\[
\pi'_{X,p}(\gamma)(\eta_1(v),\ldots ,\eta_N(v)):=(\eta_1(\sigma_1(\gamma)^{-1}v), \ldots ,\eta_N(\sigma_N(\gamma)^{-1}v)).
\]
Here $\eta(v)=(\eta_1(v),\ldots ,\eta_N(v)) \in p\textrm{-}\bigoplus_{i=1}^{N}\tilde{X}_{(p)}$. Observe that $p\textrm{-}\bigoplus_{i=1}^{N}\tilde{X}_{(p)}$ is isomorphic to $\tilde{X}_{(p)}$ itself and that with this identification $\pi_{X,p}'$ is a subrepresentation of $\pi_{X,p}$ (this is the reason why we take the $p$-stabilization of $X$). Therefore we have 
\[
\kappa_{X,p}(\Gamma,H,S) \leq  \inf_{0\ne \eta \in \ell_{p,0}(\Gamma/H, p\textrm{-}\bigoplus_{i=1}^{N}\tilde{X}_{(p)})}\sup_{s\in S}\frac{\| \pi_{X,p}'(s)\xi -\xi\|}{\|\xi\|}.
\]

Define $\kappa'_{X,p}(\Gamma,H,S):= \inf_{0 \ne \zeta\in \ell_{p,0}(\Gamma/H, \tilde{X}_{(p)})}( \sum_{s\in S} \frac{\|\pi_{X,p}(s)\zeta -\zeta\|^p}{\|\zeta\|^p})^{1/p}$ 
and for any $\epsilon>0$ get $0\ne \xi:=\xi_{\epsilon} \in \ell_{p,0}(\Gamma/H, \tilde{X}_{(p)})$ which attains $\kappa'_{X,p}$ up to an  error of $+\epsilon$. Set $\eta:=\xi\oplus \cdots \oplus \xi \in \ell_{p,0}(\Gamma/H,p\textrm{-}\bigoplus_{i=1}^{N}\tilde{X}_{(p)})$. Then by letting $\epsilon \searrow 0$, we obtain that $(\frac{|S|}{\nu})^{1/p} \kappa_{X,p}(\Gamma,H,S) \leq \kappa'_{X,p}(\Gamma,H,S)$. This together with the proof of Lemma~\ref{LemmaDisp} ends our proof.
\end{proof}

\begin{example}\label{ExampleHamming}
For $H_n=\mathrm{Cay}((\mathbb{Z}/2\mathbb{Z})^n,S_n)$, the $n$-dimensional Hamming cube ($S_n$ is the standard set of generators), we can have $\nu=1$. Therefore Theorem~\ref{TheoremPZ} implies that for any $(X,p)$, $\lambda_1(H_n;X,p)= n \cdot \kappa_{X,p}((\mathbb{Z}/2\mathbb{Z})^n,\{e\},S_n)^p$.

Similarly, for $\Gamma=\mathrm{SL}_n(\mathbb{Z})$ and $T_n$ the standard set of generators (the set of all unit elementary matrices. See Proposition 6.4.), Theorem~\ref{TheoremPZ} applies with $\nu=2$.
\end{example}

\subsection{Sphere equivalence and complex interpolation}
We give the definition of an \textit{upper modulus of continuity}.

\begin{definition}\label{DefinitionModulus}
Let $X,Y$ be Banach spaces and $\phi\colon S(X)\to S(Y)$ be a uniformly continuous map between unit spheres. 
\begin{enumerate}[$(i)$]
  \item Define $\mathcal{M}_{\phi}$ to be the class of all functions $\delta\colon [0,2]\to \mathbb{R}_{\geq 0}$ which satisfy the following conditions:
   \begin{itemize}
    \item $\delta$ is nondecreasing;
    \item $\lim_{\epsilon \to +0} \delta(\epsilon)=0$;
    \item for any $x_1,x_2\in S(X)$ with $\|x_1-x_2\|_X\leq  \epsilon$, we have $\| \phi(x_1)-\phi(x_2)\|_Y \leq \delta(\epsilon)$.
   \end{itemize}
We call an element $\delta$ in $\mathcal{M}_{\phi}$ an \textit{upper modulus of continuity} of $\phi$.
  \item Define $\overline{\phi}\colon X\to Y$ to be the extension of $\phi$ by homogeneity, namely, $\overline{\phi}(x):=\|x\|_X \phi (x/\|x\|_X)$ for $0\ne x\in X$ and $\overline{\phi}(0):=0$. We call $\overline{\phi}$ the \textit{canonical extension} of $\phi$.
\end{enumerate}
\end{definition}

Chapter~9 of \cite{BookBeLi} is an excellent reference on the sphere equivalence among Banach spaces. Here we recall some results presented in \cite{BookBeLi}. All (infinite-dimensional and separable) $\ell_p$ spaces, $L_p$ spaces, noncommutative $L_p$ spaces (recall $p\in [1,\infty)$) belong to $[\ell_2]_S$. Moreover so do all separable uniformly convex Banach spaces with unconditional basis. For $\ell_p$ spaces, the \textit{Mazur map} from $\ell_p$ to $\ell_2$ (here $\mathrm{sign}(a_i)\in S^1\subseteq \mathbb{C}$): $M_{p,2}\colon \ell_p \to \ell_2;\quad (a_i)_i \mapsto (\mathrm{sign}(a_i)|a_i|^{p/2})_i$ gives a uniform homeomorphism as follows (\cite[Theorem~9.1]{BookBeLi}) (here we only state on upper moduli of $M_{p,2}$): 
  \begin{enumerate}[$(i)$]
    \item If $p\geq 2$, then the function $\delta\colon [0,2]\to \mathbb{R}_{\geq 0}$; $\delta (\epsilon):= (p/2) \delta$ is in $\mathcal{M}_{M_{p,2}}$ ($M_{p,2}$ is Lipschitz).
    \item If $p<2$, then the function $\delta\colon [0,2]\to \mathbb{R}_{\geq 0}$; $\delta (\epsilon):= 4 \delta^{p/2}$ is in $\mathcal{M}_{M_{p,2}}$  ($M_{p,2}$ is $p/2$-H\"{o}lder).
  \end{enumerate}

Also, the complex interpolation, see a book \cite{BookBeLo}, gives examples of sphere equivalent pairs. Theorem~9.12 in \cite{BookBeLi} states that for a complex interpolation pair $(X_0,X_1)$, if either $X_0$ or $X_1$ is uniformly convex, then any $0<\theta <\theta' <1$, $X_{\theta}\sim_S X_{\theta'}$ holds.

\begin{definition}\label{DefinitionEquiv}
Let $X,Y$ be Banach spaces and $p,q\in [1,\infty)$. Let $F$ be an at most countable set. For a map $\phi \colon S(\ell_p(F,X)) \to S(\ell_q(F,Y))$, we say that $\phi$ is $\mathrm{Sym}(F)$\textit{-equivariant} if for any $\sigma \in \mathrm{Sym}(F)$, $\phi \circ \sigma_{X,p} =\sigma_{Y,q} \circ \phi$ holds true. Here a Banach space $Z$ and $r\in [1,\infty)$, the symbol $\sigma_{Z,r}$ denotes the isometry $\sigma_{Z,r}$ on $\ell_r(F,Z)$ induced by $\sigma$, namely, $(\sigma_{Z,r}\xi) (a):= \xi (\sigma^{-1}(a))$ for $\xi \in \ell_r(F,Z)$ and $a\in F$. Here by $\mathrm{Sym}(F)$ we mean the group of all permutations on $F$, \textit{including ones of infinite supports}.
\end{definition}

\begin{theorem}\label{TheoremExtra}
For any uniformly convex Banach space $X$ and $p,q \in (1,\infty)$, we have that $\tilde{X}_{(p)}\sim_S \tilde{X}_{(q)}$. Furthermore, we may have a uniform homeomorphism $\phi\colon S(\ell_p(\mathbb{N},X))\to S(\ell_q(\mathbb{N},X))$ which is $\mathrm{Sym}(\mathbb{N})$-equivariant.
\end{theorem}

\begin{proof}
Choose $1<p_0< \min \{p,q\}$ and $\infty>p_1>\max\{p,q\}$. Then \cite[Theorem~5.1.2]{BookBeLo} applies to the case where $\Omega=\mathbb{N}$ and $A_0=A_1=X$. This tells us that both of $\tilde{X}_{(p)}$ and $\tilde{X}_{(q)}$ are, respectively, isometrically isomorphic to some intermediate points of a complex interpolation pair $(\tilde{X}_{p_0},\tilde{X}_{p_1})$. Because $\tilde{X}_{p_0}$ and $\tilde{X}_{p_1}$ are uniformly convex, the aforementioned result applies.

The last assertion follows from the proof of \cite[Theorem~9.12]{BookBeLi}. Indeed, the definition of $f_x$ for $x \in \ell_p(\mathbb{N},X)$, as the minimizer of a certain norm, in Proposition~I.3 in \cite{BookBeLi} is $\mathrm{Sym}(\mathbb{N})$-equivariant in the current setting.
\end{proof}

\subsection{Sphere equivalence and $p$-stabilization}
The following proposition plays a key r\^{o}le in the proof of Theorem~\ref{mthma}. Lemma~9.9 in \cite{BookBeLi} shows the assertions in this proposition with  worse-order estimations of upper moduli.

\begin{proposition}\label{PropositionSt}
Assune that  $\phi\colon S(X)\to S(Y)$ is a uniformly continuous map for two Banach spaces $X$ and $Y$. Then for any $p\in [1,\infty)$, the map
\[
\Phi=\Phi_p \colon S(\tilde{X}_{(p)})\to S(\tilde{Y}_{(p)});\quad 
(x_i)_i \mapsto (\overline{\phi}(x_i))_i
\]
is again a uniformly continuous map that is $\mathrm{Sym}(\mathbb{N})$-equivariant. Here $\overline{\phi}$ is the canonical extension of $\phi$ and we see $\tilde{X}_{(p)}$ and $\tilde{Y}_{(p)}$, respectively, as $\ell_p(\mathbb{N},X)$ and $\ell_p(\mathbb{N},Y)$. Moreover, if $\delta(t):=Ct^{\alpha} \in \mathcal{M}_{\phi}$ for some $C>0$ and some $\alpha\in (0,1]$, then $\delta'(t):=(2C+2)t^{\alpha}$ belongs to $\mathcal{M}_{\Phi_p}$.
\end{proposition}

\begin{proof}
By construction, this $\Phi_p$ is coordinatewise and hence in particular $\mathrm{Sym}(\mathbb{N})$-equivariant. Our proof of the uniform continuity of $\Phi_p$ consists of two cases.

\noindent
\underline{\textit{Case} $1:$ \textit{for $p=1$.}}\quad Take $\delta_0\in \mathcal{M}_{\phi}$. We can replace $\delta_0$ with $\delta \in \mathcal{M}_{\phi}$ with $\delta\geq \delta_0$ such that $\delta$ is continuous and  concave in the broad sense (to do this, consider the convex hull of the graph of $\delta$). Let $(x_i)_i$ and $(y_i)_i$ be in $S(\tilde{X}_{(1)})$. 

First we consider the case where for all $i\in \mathbb{N}$ $\|x_i\|_X=\|y_i\|_X$. Set $r_i:=\|x_i\|_X$ and $\epsilon_i r_i= \|x_i-y_i\|_X $. Then by the concave inequality, we have the following:
\begin{align*}
\| \Phi((x_i)_i)-\Phi((y_i)_i)\|_{\tilde{Y}_{(1)}}\leq \sum_i r_i\delta (\epsilon_i)  \leq \delta (\sum_{i}r_i\epsilon_i)= \delta (\|(x_i)_i-(y_i)_i\|_{\tilde{X}_{(1)}}).
\end{align*}

Secondly we deal with the general case. For $(x_i)_i,(y_i)_i \in \tilde{X}_{(1)}$, define $z_i:= \frac{\|x_i\|_X}{\|y_i\|_X}y_i$ ($z_i:=x_i$ if $y_i=0$). Suppose $\| (x_i)_i-(y_i)_i\|_{\tilde{X}_{(1)}} \leq \epsilon$. Because for any $i$, $\|x_i-y_i\|_{X} \geq \|z_i-y_i\|_X$, we have that $\| (z_i)_i-(y_i)_i\|_{\tilde{X}_{(1)}} \leq \epsilon$. Hence we obtain that $\| (x_i)_i-(z_i)_i\|_{\tilde{X}_{(1)}} \leq 2\epsilon$. Therefore in the first argument, we have that $\| \Phi((x_i)_i)-\Phi((z_i)_i)\|_{\tilde{Y}_{(1)}} \leq \delta(2\epsilon)$. Since $\| \Phi((y_i)_i)-\Phi((z_i)_i)\|_{\tilde{Y}_{(1)}} \leq \epsilon$ by homogeneity, we conclude that $\delta'(t):=\delta(2t)+t$ belongs to  $\mathcal{M}_{\Phi_1}$.

\ 

\noindent
\underline{\textit{Case} $2:$ \textit{for general} $p>1$.}\quad Take $\delta_0\in \mathcal{M}_{\phi}$. In this case we need to replace $\delta_0$ with concave $\delta \in \mathcal{M}_{\phi}$ from above which in addition satisfies the following property: there exists $D>0$ such that for any $t\in [0,2^{1/p}]$, we have that $\delta(t)^p \leq D \delta (t^p)$ (this replacement is always possible by some  differential calculation). 

Once this replacement has been done, the remaining argument goes along a similar line to one in Case $1$. Thus we can show that $\delta'(t):= \left(D \delta ((2t)^p) \right)^{1/p} +t$ belongs to $\mathcal{M}_{\Phi_p}$. 

Finally, suppose that $\delta(t):=Ct^{\alpha}\in \mathcal{M}_{\phi}$. Then we do not need to replace $\delta$ because $\delta(t)^p=C^{p-1}\delta(t^p)$. Then we have that $\delta'(t):=(2 C+2)t^\alpha \geq 2^\alpha Ct^\alpha +t$  is in $\mathcal{M}_{\Phi_p}$. 
\end{proof}

\section{Proof of Theorem~\ref{mthma}}\label{SectionProofA}
The goal of this section is to prove the following theorem.

\begin{theorem}[Detailed form of Theorem~$\mathrm{\ref{mthma}}$]
\label{TheoremSEgap}
Let $X \sim_S Y$ and $\phi\colon S(X) \to S(Y)$ be a uniform homeomorphism.  
\begin{enumerate}[$(i)$]
\item For a finite graph $G=(V,E)$ and for $p\in [1,\infty)$, we have the following inequality: 
\[
\lambda_1(G;X,p)\geq \frac{1}{2}\left\{\delta_1^{-1}\left(\frac{1}{2}\left(\frac{2}{\Delta(G)}\right)^{1/p}\delta_2^{-1}\left(\frac{1}{2}\right) \lambda_1(G;Y,p)^{1/p}\right) \right\}^p.
\]
Here $\delta_1 \in \mathcal{M}_{\Phi_p}$ and $\delta_2 \in \mathcal{M}_{\Phi_p^{-1}}$ as in Proposition~$\ref{PropositionSt}$, and for $j=1,2$ if $\delta_j$ is not injective or surjective, then set $\delta_j^{-1}(s):= \inf \{t\in [0,2]: \delta_j(t)\geq s\}$. 

In particular, if $\delta_0(t):=Ct^{\alpha} \in \mathcal{M}_{\phi}$ for some $C>0$ and some $\alpha \in (0,1]$, then there exists a constant $W=W_p>0$ depending only on $\delta_2^{-1}(1/2)$ for $\delta_2 \in \mathcal{M}_{\Phi_p^{-1}}$ such that $\lambda_1(G;X,p) \geq  \frac{W^{p/\alpha}}{(C+1)^{p/\alpha}\Delta(G)^{1/\alpha}} \lambda_1(G;Y,p)^{1/\alpha}$. In short, in this case,
\[
\lambda_1(G;X,p) \succsim_{W,C,\Delta(G),p}  \lambda_1(G;Y,p)^{1/\alpha}.
\]
\item Assume that $G$ is of the form $G=\mathrm{Sch}(\Gamma,H,S)$ and that $\nu$ is as in Theorem~$\ref{TheoremPZ}$. If $\delta_0(t):=Ct^{\alpha} \in \mathcal{M}_{\phi}$ for some $C>0$ and some $\alpha \in (0,1]$, then we have that 
\[
\lambda_1(G;X,p) \succsim_{W,C,p}  \frac{1}{\nu |S|^{1/\alpha-1}} \lambda_1(G;Y,p)^{1/\alpha}.
\]
Here $W$ only depends on $\delta_2^{-1}(1/2)$ for $\delta_2 \in \mathcal{M}_{\Phi_p^{-1}}$.
\end{enumerate}
\end{theorem}

\subsection{Key proposition for Schreier coset graphs}
This part is based on a work of Bader--Furman--Gelander--Monod \cite{BFGM}. See Section~4.a in \cite{BFGM} for the original idea of them. We will show the following proposition concerning the $p$-displacement constants and $p$-$(\tau)$-type constants.

\begin{proposition}\label{PropositionTau}
Let $X \sim_S Y$ and $\phi\colon S(X) \to S(Y)$ be a uniform homeomorphism. Let $\Gamma$ be a finitely generated group, $S\not\ni e$ be a symmetric finite subset, and $H$ be a subgroup of $\Gamma$ of finite index. Then for any $p\in [1,\infty)$, we have the following inequality:
\[
\kappa_{X,p}(\Gamma, H,S) \geq \delta_1^{-1}\left( \frac{1}{2}\delta_2^{-1}\left(\frac{1}{2}\right) \kappa_{Y,p}(\Gamma, H,S)\right).
\]
Here $\delta_1\in \mathcal{M}_{\Phi_p}$ and $\delta_2\in \mathcal{M}_{\Phi_p^{-1}}$.
\end{proposition}

\begin{remark}\label{Remarktau}
Proposition~\ref{PropositionTau} in particular implies that $\kappa_{X,p}^{(\tau)}(\Gamma,S)>0$ if and only if $\kappa_{Y,p}^{(\tau)}(\Gamma, S)>0$ for a finitely generated group $\Gamma=\langle S\rangle$.
\end{remark}

\begin{proof}
By Proposition~\ref{PropositionSt}, $\Phi_p\colon \tilde{X}_{(p)}\to \tilde{Y}_{(p)}$ is a uniform homeomorphism that is $\mathrm{Sym}(\mathbb{N})$-equivariant. By coordinate transformation, we may regard $\Phi_p$ as 
\[
\Phi_p\colon S(\ell_p(\Gamma/H, \tilde{X}_{(p)})) \to S(\ell_p(\Gamma/H, \tilde{Y}_{(p)}))
\]
(note that $\ell_p(\Gamma/H, \tilde{X}_{(p)}) \simeq \tilde{X}_{(p)}$), which is $\mathrm{Sym}(\Gamma/H)$-equivariant. We thus have that $\Phi_p\circ \pi_{X,p}=\pi_{Y,p}\circ \Phi_p$. Note that we consider $\pi_{X,p}$ and $\pi_{Y,p}$ as $\Gamma$-representations, respectively, on $\ell_p(\Gamma/H, \tilde{X}_{p})$ and $\ell_p(\Gamma/H, \tilde{Y}_{p})$, \textit{not} on $\ell_{p,0}$. 

Choose any $\xi \in S(\ell_{p,0}(\Gamma/H, \tilde{X}_{(p)}))\subseteq S(\ell_{p}(\Gamma/H, \tilde{X}_{(p)}))$ and set $\eta:= \Phi_p (\xi) \in S(\ell_{p}(\Gamma/H, \tilde{Y}_{(p)}))$. We warn that $\eta $ does \textit{not} belong to $S(\ell_{p,0}(\Gamma/H, \tilde{Y}_{(p)}))$ in general. We however overcome this difficulty in the following argument. Recall that $\ell_{p}(\Gamma/H, \tilde{X}_{(p)})$ is decomposed as the direct sum of $\ell_{p}(\Gamma/H, \tilde{X}_{(p)})^{\pi_{X,p}(\Gamma)}$ and $\ell_{p,0}(\Gamma/H, \tilde{X}_{(p)})$. Note that the former subspace is sent to $\ell_{p}(\Gamma/H, \tilde{Y}_{(p)})^{\pi_{Y,p}(\Gamma)}$ by $\Phi_p$ (again because $\Phi_p$ is $\mathrm{Sym}(\Gamma/H)$-equivariant). We claim that $\mathrm{dist}(\xi, \ell_{p}(\Gamma/H, \tilde{X}_{(p)})^{\pi_{X,p}(\Gamma)})\geq \frac{1}{2}$. Indeed, since the $p$-mean of the norm is at least the norm of the mean, we first have that $\|\xi -\zeta\| \geq \|\zeta\|$ for any $\zeta \in \ell_{p}(\Gamma/H, \tilde{X}_{(p)})^{\pi_{X,p}(\Gamma)}$. Then we conclude that 
\[
2\cdot\inf_{\zeta \in \ell_{p}(\Gamma/H, \tilde{X}_{(p)})^{\pi_{X,p}(\Gamma)}} \|\xi-\zeta\|\geq \inf_{\zeta \in \ell_{p}(\Gamma/H, \tilde{X}_{(p)})^{\pi_{X,p}(\Gamma)}}(\|\xi-\zeta\|+\|\zeta\|) \geq \|\xi\|=1.
\]
In particular, from this claim we have that $\mathrm{dist}(\xi, S(\ell_{p}(\Gamma/H, \tilde{X}_{(p)})^{\pi_{X,p}(\Gamma)}))\geq \frac{1}{2}$. Therefore, by the uniform continuity of $\Phi_p^{-1}$, we have that $\mathrm{dist}(\eta, S(\ell_{p}(\Gamma/H, \tilde{Y}_{(p)})^{\pi_{Y,p}(\Gamma)})) \geq \delta_2^{-1}\left(\frac{1}{2}\right)$. 

Decompose $\eta$ as $\eta=\eta_1+\eta_0$ where $\eta_1\in \ell_{p}(\Gamma/H, \tilde{Y}_{(p)})^{\pi_{Y,p}(\Gamma)}$ and $\eta_0 \in \ell_{p,0}(\Gamma/H, \tilde{Y}_{(p)})$. We claim that 
\[
\|\eta_0\| \geq \frac{1}{2}\delta_2^{-1}\left(\frac{1}{2}\right).
\]
Indeed, let $\eta_1':= \eta_1'/\|\eta_1'\|$ (if $\eta_1=0$, then set $\eta_1'$ as any vector in $S(\ell_{p}(\Gamma/H, \tilde{Y}_{(p)})^{\pi_{Y,p}(\Gamma)})$). Then by the inequality in the paragraph above, we have that $\|\eta -\eta_1'\| \geq \delta_2^{-1}\left(\frac{1}{2}\right)$. Because $\|\eta_1 \| \geq 1- \|\eta_0\|$, we also have that $\|\eta_1-\eta_1' \| \leq \|\eta_0\|$ and that $\|\eta -\eta_1' \| \leq \|\eta -\eta_1 \| +\|\eta_1-\eta_1'\| \leq 2\|\eta_0\|$. By combining these inequalities, we prove the claim.

By the definition of $\kappa_{Y,p}(\Gamma,H,S)$, we have that 
\[
\sup_{s\in S} \| \pi_{Y,p}(s)\eta -\eta\|=\sup_{s\in S} \| \pi_{Y,p}(s)\eta_0 -\eta_0\| \geq \|\eta_0\| \kappa_{Y,p}(\Gamma,H,S)  \geq \frac{1}{2}\delta_2^{-1}\left(\frac{1}{2}\right) \kappa_{Y,p}(\Gamma,H,S) .
\]
Finally, because $\Phi_p\circ \pi_{X,p}=\pi_{Y,p}\circ \Phi_p$, we conclude by the uniform continuity of $\Phi_p$ that 
\[
\sup_{s\in S} \| \pi_{X,p}(s)\xi -\xi\|\geq \delta_1^{-1}\left(\frac{1}{2}\delta_2^{-1}\left(\frac{1}{2}\right)\kappa_{Y,p}(\Gamma,H,S) \right).
\]
By taking the infimum over $\xi \in S(\ell_{p,0}(\Gamma/H,\tilde{X}_{(p)}))$, we obtain the assertion.
\end{proof}

\subsection{Proof of Theorem~\ref{TheoremSEgap}: the Gross trick}\label{subsection:Gross}

\begin{proof}[Proof of Theorem~$\ref{TheoremSEgap}$]

We shall divide the proof into two cases.

\noindent
\underline{\textit{Case $1:$ when $G$ is a Schreier coset graph.}}\quad In this case, Proposition~\ref{PropositionTau} and Lemma~\ref{LemmaDisp} ends the proof of item $(i)$ (with a slightly better estimate for constants). To prove item $(ii)$, also employ Theorem~\ref{TheoremPZ}. 

\ 

\noindent
\underline{\textit{Case $2:$ in general case.}}\quad We shall appeal to a Theorem by Gross \cite{Gross} that any finite connected and \textit{regular} graph (possibly with multiple edges and self-loops) with \textit{even} degree can be realized as  a Schreier coset graph. The following argument is the \textit{Gross trick}.

Let $G=(V,E)$ be a finite connected graph. Then we take the \textit{even regularization} of $G$ in the following sense: we let $V$ unchanged. We first double each edge in $E$. Note that then for any $v,w \in V$, $\mathrm{deg}(v)-\mathrm{deg}(w) \in 2\mathbb{Z}$ and that the maximum degree is $2\Delta(G)$. Finally, we let a vertex $v$ whose degree is $2\Delta(G)$ unchanged, and for all the other vertices add, respectively, appropriate numbers of self-loops to have the resulting degree $=2\Delta(G)$ for each vertex. We write the resulting graph as $G'=(V,E')$. Then by the Gross theorem, $G'$ can be realized as a Schreier coset  graph and thus Case $1$ applies to $G'$. Finally observe that $\lambda_1(G';Z,p)=2\lambda_1(G;Z,p)$ for any $Z$ because self-loops do not affect the spectral gap.
\end{proof}

\section{Proofs of Theorem~\ref{mthmb} and Corollary~\ref{mcorc}}\label{Sectionproof}

\begin{proof}[Proof of Theorem~$\ref{mthmb}$]

Let $X\sim_S Y $, where $Y$ is uniformly convex and let $p,q\in (1,\infty)$. By Theorem~\ref{TheoremExtra}, there exists an $\mathrm{Sym}(\mathbb{N})$-equivariant  uniform homeomorphism $\Phi:=\Phi_{p,q}\colon S(\tilde{Y}_{(p)})\to S(\tilde{Y}_{(q)})$. First we start from the case where $G$ is of the form $\mathrm{Sch}(\Gamma,H,S)$. Then we regard $\Phi$ as an $\mathrm{Sym}(\Gamma/H)$-equivariant uniform homeomorphism
\[
\Phi\colon S(\ell_p(\Gamma/H, \tilde{Y}_{(p)})) \to S(\ell_q(\Gamma/H, \tilde{Y}_{(q)})).
\]
We thus may apply a similar argument to Proposition~\ref{PropositionTau} to the pair $((Y,p);(Y,q))$. Because Proposition~\ref{PropositionTau} works for the pairs $(((X,p);(Y,p)))$ and $((Y,q);(X,q))$, we are done.

For general cases, apply the Gross trick.
\end{proof}

\begin{proof}[Proof of Corollary~$\ref{mcorc}$]

The first assertion holds true by Theorem~\ref{mthma}, Theorem~\ref{mthmb}, and the fact of that uniformly curved Banach spaces are isomorphic (and in particular sphere equivalent) to some uniformly convex Banach spaces, see Section~1. The second assertion holds true for the following reason: if $X\in [\ell_2]_S$, then by Theorem~\ref{mthma} and Lemma~\ref{LemmaSpGap}, the $(X,p)$-ander property is equivalent to the $(\mathbb{R},p)$-ander property. The original Matou\v{s}ek extrapolation enables us to extend our results even for $p=1$.
\end{proof}

\section{Applications}\label{SectionAppli}
\subsection{Distortions of expander graphs}
The following corollary is an immediate byproduct of Corollary~\ref{mcorc} and Corollary~\ref{CorollaryDist}. This result generalizes the result of Matou\v{s}ek \cite{Matousek} for $\ell_p$ spaces in qualitative sense.

\begin{corollary}\label{CorollaryExpDist}
Let $\{G_n\}_n$ be a family of expanders. Then for any infinite-dimensional Banach space $X$ which is sphere equivalent to a uniformly curved Banach space, we have that $c_X(G_n) \asymp_X \mathrm{diam}(G_n)$. 
\end{corollary}
\subsection{Spectral gaps of $\ell_p$ spaces with different exponent}

By utilizing Theorem~\ref{TheoremSEgap}, we have certain estimation of the $\ell_p$ spectral gap with exponent $q$, where $p\ne q$. Below we only give a result for $p,q\geq 2$, but we have results for all pairs $(p,q)$.
\begin{corollary}\label{Corollarylplq}
Let $p,q\in [2,\infty)$. Then for any finite connected graph $G$, we have 
\[
\lambda_1(G;\ell_p,q) \asymp_{p,q,\Delta(G)} \lambda_1(G;\mathbb{R},2)^{q/2}.
\]
\end{corollary}

We note that the Mazur map $M_{p,q}\colon S(\ell_p) \to S(\ell_q)$ is Lipschitz if and only if $p\geq q$. The assertion above nevertheless states that even for $q\geq p\geq 2$, the $q$-th root of the spectral gap has the same order as the square root of the classical one.

\begin{proof}
Because $p\geq 2$, the Mazur map $M_{p,2}$ is Lipschitz. Theorem~\ref{TheoremSEgap} implies that
\[
\lambda_1(G;\ell_p,q)^{1/q} \succsim_{p,q,\Delta(G)} \lambda_1(G;\ell_2,q)^{1/q}=\lambda_1(G;\mathbb{R},q)^{1/q}.
\]
Then use Matou\v{s}ek's extrapolation for $q$ to obtain that $\lambda_1(G;\mathbb{R},q)^{1/q} \asymp_{q,\Delta(G)} \lambda_1(G;\mathbb{R},2)^{1/2}$. To have the converse-order inequality, note that $\lambda_1(G;\ell_p,q)\leq \lambda_1(G;\mathbb{R},q)$.
\end{proof}

\subsection{The order of distortions of Hamming cubes into $\ell_p$ space}
As a corollary to Theorem~\ref{TheoremSEgap}, here we give a different proof of the following assertion, which is well-known for $p\in [1,2]$ and has been first proved by Naor--Schechtman \cite{NS} (see the introduction of their paper) for $p>2$.
\begin{corollary}$($Compare with \cite{NS}$)$\label{CorollaryHam}
Let $H_n$ be the $n$-dimensional Hamming cube. 
\begin{enumerate}[$(1)$]
  \item For $p\in [1,2)$, then we have that $\lambda_1(H_n;\ell_p,p)\asymp_p 1$, $c_{\ell_p}(H_n) \asymp_p n^{1-\frac{1}{p}}$ and that $\sup\{ \alpha \in [0,1]: t^{\alpha} \in \mathcal{R}_{\ell_p}(\coprod_n H_n)\}=\frac{1}{p}$.
  \item For $p\in [2,\infty)$, then we have that $\lambda_1(H_n;\ell_p,p)\asymp_p n^{1-\frac{p}{2}}$, $c_{\ell_p}(H_n) \asymp_p n^{\frac{1}{2}}$ and that $\sup\{ \alpha \in [0,1]: t^{\alpha} \in \mathcal{R}_{\ell_p}(\coprod_n H_n)\}=\frac{1}{2}$.
\end{enumerate}
\end{corollary}

\begin{proof}
Let $S_n$ be the standard set of generators of $\Gamma_n:=(\mathbb{Z}/2\mathbb{Z})^n$. Then by Proposition~\ref{PropositionTau} (with the usual Mazur map), we have the following order inequalities:
\begin{itemize}
 \item For $p\in [1,2)$, $\kappa_{\ell_p,p}(\Gamma_n,\{e\},S_n) \succsim_{p} \kappa_{\ell_2,2}(\Gamma_n,\{e\},S_n)^{2/p} \asymp n^{-1/p}$.
 \item For $p\in [2,\infty)$, $\kappa_{\ell_p,p}(\Gamma_n,\{e\},S_n) \succsim_{p} \kappa_{\ell_2,2}(\Gamma_n,\{e\},S_n) \asymp n^{-1/2}$.
\end{itemize}
Here we can obtain that $\kappa_{\ell_2,2}(\Gamma_n,\{e\},S_n) \asymp n^{-1/2}$ by spectral theory. First we consider the case where $p\in [1,2)$. Then for each $n$, we can find a vector $\xi_n \in S(\ell_{p,0}(\Gamma_n ,\mathbb{R})) $ which realizes the order $n^{-1/p}$ above, namely, which satisfies that
\[
\sup_{s\in S_n} \| \pi_{\ell_p,p}(s) \xi_n -\xi_n\| \asymp_p  n^{-1/p}.
\]
Therefore $\kappa_{\ell_p,p}(\Gamma_n,\{e\},S_n)  \asymp_p n^{-1/p}$. This together with the argument in Example~\ref{ExampleHamming} and Theorem~\ref{TheoremJoVa} implies the order equality on $\lambda_1$ and one side ($\succsim $) of the order inequality in the distortion estimation in the corollary. Because the standard embedding of $H_n$ into the $n$-dimensional real $\ell_p$ space achieves the order, this distortion order is optimal. For the compression exponent, the same embedding gives the exponent by Lemma~\ref{LemmaAustin}.

Finally we deal with the case where $p\in [2,\infty)$. Then we observe  that $H_n \hookrightarrow \ell_2 \hookrightarrow L_p$ (for the second one we take an isometric embedding) and argue in a similar way to the one above.
\end{proof}

Similarly, we have the following estimations for the $p$-$(\tau)$-type constant on $\ell_p$  for $\Gamma_n=\mathrm{SL}_n(\mathbb{Z})$ with a standard set of generators $T_n$ for $n\geq 3$, with the aid of a theorem of M. Kassabov \cite{Kassabov} on the Kazhdan constant of $(\Gamma_n,T_n)$. 

\begin{proposition}[Compare with \cite{Kassabov}]\label{PropositionKassabov}
For $n\geq 3$, let $T_n$ be the set of all unit elementary matrices in $\mathrm{SL}_n(\mathbb{Z})$, namely, the set of all matrices in $\mathrm{SL}_n(\mathbb{Z})$ whose on-diagonal entries are $1$, all but one off-diagonal entries are $0$, and the remaining entry has a value in $\{\pm 1\}$. Then we have the following:
\begin{enumerate}[$(1)$]
   \item For $p\in [1,2)$, $\kappa^{(\tau)}_{\ell_p,p} (\mathrm{SL}_n(\mathbb{Z}),T_n)\asymp_{p} n^{-1/p}$.
     \item For $p\in [2,\infty)$, $\kappa^{(\tau)}_{\ell_p,p} (\mathrm{SL}_n(\mathbb{Z}),T_n)\asymp_{p} n^{-1/2}$.
\end{enumerate}
\end{proposition}

Note that Theorem~\ref{TheoremPZ} applies to the pair $(\mathrm{SL}_n(\mathbb{Z}),T_n)$ above with $\nu=2$. Hence we have a corresponding estimation of Banach spectral gaps for some Schreier graphs coming from $(\mathrm{SL}_n(\mathbb{Z}),T_n)$. (For instance, for $\mathrm{Cay}(\mathrm{SL}_n(\mathbb{Z}/k \mathbb{Z}),T_n\ \mathrm{mod}\ k)$ for $k\geq 2$. The estimates above does not depend on $k$.) We also mention that the same-order estimates as in Proposition~\ref{PropositionKassabov} hold for $L_p$-Kazhdan constants for $p\in (1,\infty)$. Here for a fixed $p\in (1,\infty)$,   the $L_p$-Kazhdan constant for $(\Gamma,S)$, $\Gamma=\langle S\rangle$, is defined by
\[
\inf_{(\rho,B)}\inf_{\xi \in S(B_{\rho(\Gamma)}')} \sup_{s\in S} \|\rho(s)\xi -\xi\|,
\]
where $(\rho,B)$ runs over all isometric linear representations of $\Gamma$ on $B$, where  $B$ is of the form $L_p(\Omega,\mu)$ for an arbitrary measure space $(\Omega,\mu)$ (recall that $p$ is fixed). Here $B_{\rho(\Gamma)}'$, defined in \cite[Proposition~2.6]{BFGM},  denotes the natural complement of the space $B^{\rho(\Gamma)}$ of all $\rho(\Gamma)$-invariant vectors in $B$ (note that $B$ is uniformly convex and uniformly smooth because $p\in (1,\infty)$).

\section*{acknowledgments}
The author wishes to express his gratitude to Professor Alain Valette for hospitality for his stay at University of Neuch\^{a}tel from April, 2012 to September, 2012, and for encouragement for this study. He also thanks Professor Guoliang Yu and Professor Qin Wang for their kindness invitation to East Normal China University and Fudan University in Shanghai in July, 2013. Part of this work was done during these stays. He acknowledges Assaf Naor for comments and references. He is also grateful to Kei Funano, Takefumi Kondo, and Narutaka Ozawa for discussions, to Masahiko Kanai and Akihiro Munemasa for comments, and to the anonymous referee, whose comments considerably improve this paper. 

\bibliographystyle{amsalpha}
\bibliography{mimura_IMRN_2014_137.bib}

\end{document}